\documentclass[11pt]{amsart}
\usepackage[centering,margin=1.1in]{geometry}
\usepackage{graphicx}
\usepackage{amssymb}
\usepackage{epstopdf}
\usepackage{amsmath}
\usepackage{mathrsfs}
\usepackage[all]{xy}
\usepackage{color}
\usepackage[colorlinks=true, pdfstartview=FitV,linkcolor=blue,citecolor=blue,urlcolor=blue]{hyperref}
\usepackage{enumerate, amsmath, amsthm, amsfonts, amssymb, 
mathrsfs, graphicx, paralist, ulem}

\usepackage{tikz-cd}

\newtheorem {theorem}    {Theorem}[section]
\newtheorem {problem}    {Problem}
\newtheorem {lemma}      [theorem]    {Lemma}
\newtheorem {cor}  [theorem]    {Corollary}
\newtheorem {prop}[theorem]    {Proposition}
\numberwithin{equation}{section}
\newtheorem{conj}[theorem]{Conjecture}

\theoremstyle{definition}
\newtheorem {rmk}   [theorem] {Remark}
\newtheorem {rmks}   [theorem] {Remarks}
\newtheorem {exa}   [theorem] {Example}

\newcommand{\cal}{\mathcal}

\newcommand{\scr}{\mathscr}

\newcommand{\w}{{\mathfrak{w}}}
\newcommand{\la}{\lambda}
\newcommand{\lap}[2]{\la _{#1,#2}}
\newcommand{\linch}[2]{\la({{\mathscr{C}}_{#1,#2}^+})}
\newcommand{\lambdaCplus}[2]{\la({{\mathscr{C}}_{#1,#2}^+})}

\newcommand{\lambdaCminus}[2]{\la({\mathscr{C}}^-_{#1,#2})}
\newcommand{\inch}[2]{{{\mathscr{C}}_{#1,#2}^+}}
\newcommand{\dech}[2]{{\mathscr{C}}^-_{#1,#2}}
\newcommand{\ldechr}[2]{\la({\mathscr{C}}^-_{#1,#2})^{\ast}}
\newcommand{\picsize}{2}

\newcommand{\fw}{\mathfrak{w}}

\begin{document}

\title{Whittaker functions and Demazure characters}
\date{\today}


\author[K.-H. Lee]{Kyu-Hwan Lee$^{\star}$}
\thanks{$^{\star}$K.-H.L. was partially supported by a grant from the Simons Foundation (\#318706).}
\address{Department of
Mathematics, University of Connecticut, Storrs, CT 06269, U.S.A.}
\email{khlee@math.uconn.edu}

\author[C. Lenart]{Cristian Lenart$^{\dagger}$}
\thanks{$^{\dagger}$ C.L. was partially supported by the NSF grant DMS--1362627.}
\address{Department of
Mathematics and Statistics, State University of New York at Albany, Albany, NY 12222, U.S.A.}
\email{clenart@albany.edu}

\author[D. Liu]{Dongwen Liu, \\ with Appendix by Dinakar Muthiah and Anna Pusk\'as}
\address{School of Mathematical Science, Zhejiang University, Hangzhou 310027, Zhejiang,
P.R. China }
\email{maliu@zju.edu.cn}

\address{Department of Mathematical and Statistical Sciences,
University of Alberta, Edmonton, AB, Canada T6G 2G1}
\email{muthiah@ualberta.ca}

\address{Department of Mathematical and Statistical Sciences,
University of Alberta, Edmonton, AB, Canada T6G 2G1}
\email{puskas@ualberta.ca}


\begin{abstract} 
In this paper,  we consider how to express an Iwahori--Whittaker function through Demazure characters. Under some interesting combinatorial conditions, we obtain an explicit formula and thereby a generalization of the Casselman--Shalika formula. Under the same conditions, we compute the transition matrix between two natural bases for the space of Iwahori fixed vectors of an induced representation of a $p$-adic group; this generalizes a result of Bump--Nakasuji.
\end{abstract}

\maketitle

\section{Introduction}\label{sec:intro}

The Casselman--Shalika formula describes a spherical Whittaker function using the root system and the character of an irreducible representation of the dual group. The formula not only  plays a fundamental role in the theory of $p$-adic groups and automorphic forms, but also connects many different constructions in mathematics, such as Schubert varieties, crystal bases and Macdonald polynomials. For example, see \cite{BBL}.

In this paper, we study a generalization of the Casselman--Shalika formula to the case of Iwahori--Whittaker functions through Demazure characters. To be precise, let $\mathfrak g$ be a finite-dimensional simple Lie algebra over $\mathbb C$, which should be considered as the Lie algebra of the dual group. Let $P$ be the weight lattice of $\mathfrak g$, and $\mathbb C[P]$ the group algebra of $P$, with basis $e^\lambda$, $\lambda \in P$. The subset of dominant weights will be denoted by $P_+$. We also denote by $\Phi\supset\Phi^+$ the set of roots and positive roots, by $\Pi=\{a_i \}_{i \in I}$ the set of simple roots, and by $S=\{ \sigma_i \}_{i \in I}$ the set of simple reflections, which generates the Weyl group $W$. Let $v$ be an indeterminate, and set $\mathcal O_v= \mathbb C(v) \otimes \mathbb C[P]$. 

Consider the Demazure character $\partial_{w, \lambda}$ for $w\in W$ and $\lambda \in P_+$, which is the formal character of the Demazure module associated with the weight $w\lambda$. When $w =w_\circ$, the longest element, the character $\partial_{w_\circ, \lambda}$ is nothing but the character of the irreducible representation of $\mathfrak g$ with highest weight $\lambda$. Now the Casselman--Shalika formula is given by \begin{equation} \label{CS} \widetilde{W}_{w_\circ, \lambda} = \left ( \prod_{\alpha \in \Phi_+} (1-v e^{-\alpha}) \right ) \partial_{w_\circ, \lambda}, \end{equation} where $\widetilde{W}_{w_\circ, \lambda}$ is the spherical Whittaker function. 

As mentioned above, this paper is concerned with generalizing the formula \eqref{CS} to the case  involving the Iwahori--Whittaker functions ${W}_{w, \lambda}$ (to be defined in the next section) and the Demazure characters $\partial_{x, \lambda}$, for $w, x \in W$. That is to say, we would like to compute the coefficients ${C}_{w,x}\in \cal O_v$, $x\leq w$, in the expansion
\[
{W}_{w,\lambda}=\sum_{x\leq w} {C}_{w,x}\partial_{x,\lambda}.
\]

To make the problem more tractable, we consider the \textit{Demazure atoms} $D_{x, \lambda}$ (see Section~\ref{sect2}),  instead of working with the Demazure characters directly. 
We write
\[
{W}_{w,\lambda}=\sum_{x\leq w} c_{w,x} D_{x, \lambda},
\]
and study how to compute $c_{w,x}\in \mathcal O_v$, $x\leq w$. The coefficients  $C_{w,x}$ and $c_{w,x}$ are related in a simple way (Corollary \ref{simple}):
\[ c_{w,x}=\sum_{x\leq y\leq w}C_{w,y} \quad \text{ and } \quad C_{w,x}=\sum_{x\leq y\leq w}(-1)^{\ell(y)-\ell(x)}c_{w,y}.\]

Still, in general, it would be difficult to obtain a complete description of the coefficients $c_{w,x}$. However, the main result of this paper shows how to compute the coefficients $c_{w,x}$ under some interesting conditions involving {\it good words} and {\it shellability}. More precisely, under Condition (A) or (B) at the beginning of Section \ref{sec-main}, we obtain:

\begin{theorem}
Let $w=s_1 \cdots s_n$ be a reduced word with $s_i=s_{\alpha_i}$ for some $\alpha_i\in \Pi$, $i=1\ldots, n$, and 
\[
\beta_i=s_1\cdots \hat{s}_{i_1}\cdots \hat{s}_{i_2}\cdots \alpha_i,\quad i=1,\ldots, n,
\] where the indices $i_1< \dots <i_d$ between $1$ and $n$ are determined by condition {\rm (A)} or {\rm (B)}. 
 Then we have 
\[
c_{w,x}=(1-ve^{-\beta_1})\cdots {\mathcal T}_{\beta_{i_1}}\left(\cdots
{\mathcal T}_{\beta_{i_d}}\left(\cdots (1-ve^{-\beta_n})\right)\cdots\right),
\] where ${\mathcal T}_{\beta}=(1-ve^{-\beta})\partial_{\beta}-1$ and $\partial_\beta$ is the Demazure operator corresponding to the root $\beta$. 
\end{theorem}

Conditions (A) and (B) are intriguing. In fact, based on thorough computer tests, in Section~\ref{cond-a-b} we conjecture that they are equivalent in a strong sense. Shortly after posting our paper, D. Muthiah and A. Pusk\'as proved our conjecture; their proof is included as an Appendix. As discussed in Section~\ref{goodword}, Condition (A) is closely related to smoothness of Schubert varieties in flag varieties $G/B$. We also present some statistical information regarding the frequency with which these conditions are satisfied. 

We establish an application of Conditions (A) and (B) to the problem of computing the transition matrix between two natural bases for the space of Iwahori fixed vectors of an induced representation of a $p$-adic group. The same problem was studied by Bump and Nakasuji \cite{BN}. They showed that, in the simply-laced case, when $w$ admits a good word for $x$, the entry $m(x,w)$ of the transition matrix is given by 
\begin{equation} \label{eqn-m}
m(x,w)=\prod_{\alpha\in S(x,w)}\frac{1-q^{-1}{\bf z}^\alpha}{1-{\bf z}^\alpha},
\end{equation}
where $S(x,w)$ is the set of roots determined by the good word condition. However, it seems that there is a gap in the proof of \cite{BN}, which we do not know how to fix at the present. In Section \ref{sec-Cass}, we assume Condition (B) and prove the formula \eqref{eqn-m} with $S(x,w)$ replaced by a set determined by Condition (B). The main idea of the proof is similar to that of \cite{BN}. Given the equivalence of Conditions (A) and (B), the Bump-Nakasuji result in full root system generality follows. This provides another evidence that Conditions (A) and (B) are natural ones to be considered in representation theory.   

Related to the above mentioned coefficients $m(w,x)$, it is also worth noting the recent paper of Nakasuji and Naruse \cite{NN}. By using a change of basis in the Hecke algebra, they express all of these coefficients in a completely different way compared to \eqref{eqn-m}, namely as sums over combinatorial sets. The mentioned change of basis in the Hecke algebra generalizes the theory of so-called {\it root polynomials}, which provides similar combinatorial formulas for localizations of Schubert classes in the equivariant cohomology and $K$-theory of flag varieties, see \cite{LZ} and the references therein, as well as \cite[Remark~1]{NN}. 

The fact that there are two types of formulas for the coefficients $m(w,x)$, namely the general formula in \cite{NN} and the simpler formula \eqref{eqn-m} if Conditions (A) and (B) hold, is very similar to the existence of a general summation formula for Schubert classes (via root polynomials), versus a much simpler product formula in the smooth case, see \cite[Chapter~7]{BL}. It turns out that the latter formula is hard to derive from the former, so completely separate proofs are needed. In this context, it is not surprising that Conditions (A) and (B) are related to smoothness of Schubert varieties, as noted above.

\section{Description of the  problem}\label{sect2}

In this section, we present the main question of this paper, introduced in the previous section, in more detail. We keep the notions fixed in the previous section.

Recall that the Hecke algebra $\cal H_v$ is the algebra over $\mathbb C(v)$ defined by the generators $T_i$, $i \in I$, subject to the quadratic relations \[ T_i^2=(v-1)T_i +v , \qquad i\in I ,\] and the braid relations corresponding to $W$. The algebra $\cal H_v$
acts on $\cal O_v$ by
\[
T_i \mapsto \cal T_i :=(1-v e^{-a_i})\partial_i-1,\quad i\in I,
\]
where $\partial_i$, $i\in I$, are the Demazure operators defined by
\[
\partial_i =\frac{1-e^{-a_i}\sigma_i}{1-e^{-a_i}}.
\]
In particular the operators $\cal T_i$, $i\in I$, satisfy the braid relations. Hence one may define 
\[
T_w\mapsto \cal T_w=\cal T_{i_1}\cdots  \cal T_{i_l}
\]
 for an arbitrary choice of reduced expression  $w=\sigma_{i_1}\cdots \sigma_{i_l}$. For a dominant weight $\lambda\in P_+$, define 
\[
W_{w,\lambda}=\cal T_w e^\lambda \quad \text{ and } \quad \widetilde{W}_{w,\lambda}=\sum_{x\leq w} W_{x,\lambda},
 \quad w\in W.
\]
As shown in \cite{BBL}, the expression $W_{w, \lambda}$ corresponds to the Iwahori--Whittaker function, and the sum $\widetilde{W}_{w_\circ, \lambda}$ corresponds to the spherical Whittaker function where $w_\circ \in W$ is the longest element.

It is well-known that the Demazure operators 
$\partial_i$, $i\in I$, satisfy the braid relations as well so that the operator $\partial_w$ is well-defined for $w \in W$ using any reduced expression of $w$. Then the Demazure character is given by
\[
\partial_{w,\lambda}=\partial_w e^\lambda, \quad \lambda \in P_+, 
\]
which is the formal character of the Demazure module associated with the weight $w\lambda$. Recall the Casselman--Shalika formula:
\begin{equation} \label{CS-1} \widetilde{W}_{w_\circ, \lambda} = \left ( \prod_{\alpha \in \Phi_+} (1-v e^{-\alpha}) \right ) \partial_{w_\circ, \lambda}. \end{equation}

As mentioned above, we are interested in generalizing the formula \eqref{CS-1} to the cases involving $\widetilde{W}_{w, \lambda}$ (or $W_{w, \lambda}$) and $\partial_{x, \lambda}$ for $w, x \in W$. Precisely, we would like to compute the coefficients $\widetilde{C}_{w,x}\in \cal O_v$, $x\leq w$, in the expansion
\[
\widetilde{W}_{w,\lambda}=\sum_{x\leq w} \widetilde{C}_{w,x}\partial_{x,\lambda}.
\]
Alternatively, if we write 
\[
W_{w,\lambda}=\sum_{x\leq w} C_{w, x} \partial_{x,\lambda},
\]
we have
\[
\widetilde{C}_{w,x}=\sum_{x\leq y\leq w} C_{y,x}
\]
and
\[
C_{w,x}=\sum_{x\leq y\leq w}(-1)^{\ell(w)-\ell(y)}\widetilde{C}_{y,x}.
\]
by the M\"{o}bius inversion \cite[Theorem 1.2]{D1}.

However, we found it more convenient to work with {\it Demazure atoms}. We define
\[D_i=\partial_i-1=e^{-a_i}\frac{1-\sigma_i}{1-e^{-a_i}}, \quad i \in I, 
\]
which is the specialization of $\cal T_i$ at $v\to 0$. Then $D_i$, $i\in I$ satisfy the braid relations, and 
we define $D_w$, $w\in W$ in the obvious way. Now the Demazure atoms are defined to be $$D_{w,\lambda}=D_w e^\lambda \quad \text{ for } w \in W \text{ and }   \lambda\in P_+.$$ 

\begin{problem} \label{prob}
Consider the transition between $\cal T_w$ and $D_w$,
\[
\cal T_w=\sum_{x\leq w} c_{w,x} D_x,
\]
and  study how to compute $c_{w,x}\in {\bf Z}[v]\otimes{\bf Z}[P]$, $x\leq w$. 
\end{problem}

The coefficients  $C_{w,x}$ and $c_{w,x}$ can be related in a simple way, using the fact that the Demazure character is the sum of all the lower Demazure atoms. We give a proof of this fact below using a result in \cite{BBL}.

\begin{lemma}
$\partial_w=\sum_{x\leq w}D_x$ and $D_w=\sum_{x\leq w} (-1)^{\ell(w)-\ell(x)} \partial_x$.
\end{lemma}

\begin{proof}
$\partial_i$, $i\in I$ are the specialization of 
\[
\mathfrak{D}_i:=\cal T_i +1=(1-ve^{-a_i})\partial_i
\] at $v\to 0$. Let $\mathfrak{w}$ be a reduced expression of $w$, and define $\mathfrak{D}_\mathfrak{w}$ in the obvious way. By \cite[Theorem 6]{BBL} one has
\[
\mathfrak{D}_\mathfrak{w}=\sum_{x\leq w}P_{\mathfrak{w}, x}(v)\cal T_x,
\]
where $P_{x,\mathfrak{w}}$ is the Poincar\'e polynomial of fibre of the Bott--Samelson resolution $Z_\mathfrak{w}\to X_w$ over the open cell
$Y_x=BxB/B$. Specializing $v\to 0$ gives that
\[
\partial_w=\sum_{x\leq w}P_{x,\mathfrak{w}}(0)D_x=\sum_{x\leq w}D_x.
\]
 \end{proof}
 
 \begin{cor} \label{simple}
 $c_{w,x}=\sum_{x\leq y\leq w}C_{w,y}$ and $C_{w,x}=\sum_{x\leq y\leq w}(-1)^{\ell(y)-\ell(x)}c_{w,y}.$
 \end{cor}

By the reduction made above, the computation of the coefficients $\widetilde{C}_{w,x}$ or $C_{w,x}$ is  equivalent to the computation of the coefficients $c_{w,x}$, for $x\leq w$. Hence we will focus on Problem \ref{prob} from now on. 
 
Note that the operators $D_i$  are twisted derivations in the sense that
\begin{equation}\label{der}
D_i(fg)=D_i(f)\cdot g +\sigma_i(f)\cdot D_i(g),\quad f, g\in {\bf Z}[P].
\end{equation}
In fact the last equation is the specialization at $v\to 0$ of 
\begin{equation}\label{der2}
\cal T_i(fg)=(1-v)D_i(f)\cdot g +\sigma_i(f)\cdot \cal T_i(g),\quad f, g\in {\bf Z}[P].
\end{equation}

It is also known that $ T_w$, $w\in W$ satisfy the relation
\begin{equation}\label{hecke}
 T_i  \cdot T_w=\left\{\begin{array}{ll} T_{\sigma_i w}& \textrm{if } \sigma_iw>w,\\
(v-1)T_w+ v T_{\sigma_i w}& \textrm{if }\sigma_iw <w.\end{array}\right.
\end{equation}
For example one has the quadratic relation $ T_i^2=(v-1)T_i +v$, $i\in I$.
Specializing (\ref{hecke}) at $v\to 0$ gives
\begin{equation}\label{dem}
D_i \cdot D_w=\left\{\begin{array}{l} D_{\sigma_i w}\quad \textrm{if } \sigma_iw>w,\\
-D_w\quad \textrm{if }\sigma_iw <w.\end{array}\right.
\end{equation}

 \section{Induction steps}

 In this section we give some general inductive steps for later use. We  recall a well-known lemma from \cite{D1}, which is called $Z(s, w_1, w_2)$ property of the Bruhat order, and it will be used frequently in this paper.

\begin{lemma}\label{lem}
Let $s\in S$ be a simple reflection and $w_1, w_2\in W$. Assume that $w_1< sw_1$, $w_2< sw_2$. Then
\[
w_1\leq w_2 \Longleftrightarrow w_1\leq sw_2 \Longleftrightarrow sw_1\leq sw_2.
\]
\end{lemma}

This lemma can be visualized using the diamond square in Figure~\ref{fig}, where the validities of the three dashed lines are all equivalent.

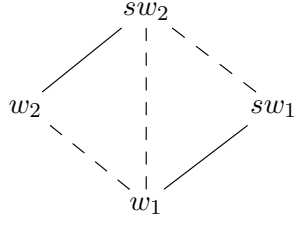
\begin{figure} 
\begin{displaymath} 
\xymatrix{ & sw_2 \ar@{--}[rd] \ar@{-}[ld] \ar@{--}[dd]  &\\ 
w_2 \ar@{--}[rd] & & sw_1 \ar@{-}[dl] \\ & w_1 & }
\end{displaymath}
\caption{$Z(s,w_1,w_2)$ property} \label{fig}
\end{figure}

The following lemma can be easily verified by using (\ref{dem}).

\begin{lemma}\label{lem2}
Let $\alpha\in\Pi$ be a simple root and $s=s_\alpha$. Then
\[
\cal T_s\cdot D_w=\left\{\begin{array}{ll} (1-ve^{-\alpha})D_{sw}-v e^{-\alpha}D_w& \textrm{if }sw>w,\\
-D_w & \textrm{if } sw<w.\end{array}\right.
\]
\end{lemma}

\begin{lemma}\label{lem3}
Assume that the simple reflection $s=s_\alpha$ is a left ascent of $w$, i.e., $sw>w$. Then 
\begin{align*}
\cal T_{sw}&= \sum_{x\leq w}{\mathcal T}_s(c_{w,x})D_x+\sum_{x\leq w, \;x<sx}(1-ve^{-\alpha})s(c_{w,x}) D_{sx}
\\&-\sum_{x\leq w, \;x>sx}(1-ve^{-\alpha})s(c_{w,x})D_x.
\end{align*}
\end{lemma}

\begin{proof}
Applying $\cal T_s$ to the equation $\cal T_w=\sum_{x\leq w}c_{w,x}D_x$ and using (\ref{der2}) gives that
\[
\cal T_{sw}=\sum_{x\leq w} s(c_{w,x}) \cal T_s \cdot D_x +(1-v)D_s(c_{w,x})D_x.
\]
The lemma follows from inserting Lemma \ref{lem2} into the last equation, and also from noting that
\begin{align*}
&(1-v)D_s-ve^{-\alpha}s=(1-ve^{-\alpha})\partial_s-1={\mathcal T}_s,\\
&(1-v)D_s-s={\mathcal T}_s-(1-ve^{-\alpha})s;
\end{align*}
here the second equation is an immediate consequence of the first. 
\end{proof}

By comparing the coefficients in Lemma \ref{lem3} with $\cal T_{sw}=\sum_{x\leq sw} c_{sw, x}D_x$, we obtain the following inductive algorithm. 
\begin{prop}\label{ind}
Assume that $w<sw$, $s=s_\alpha\in S$, and that $x\leq sw$. Then

{\rm (i)} if $x\leq w$, $x<sx$, then
\[
c_{sw, x}={\mathcal T}_s (c_{w,x});
\]

{\rm (ii)} if $x\leq w$, $x>sx$, then
\[
c_{sw,x}=(1-ve^{-\alpha})s(c_{w,sx}-c_{w,x})+{\mathcal T}_s(c_{w,x});
\]

{\rm (iii)} if $x\not\leq w$, in which case $x>sx$, then 
\[
c_{sw,x}=(1-ve^{-\alpha}) s(c_{w,sx}).
\]
\end{prop}

The three cases are illustrated in Figure~\ref{fig-2}. Note that in the last case we have either $x$ and $w$ incomparable, as depicted, or $x=sw>w=sx$.

\begin{figure}
\begin{displaymath}
\xymatrix{ & sw \ar@{-}[rd] \ar@{-}[ld] \ar@{-}[dd]  &\\ 
w \ar@{-}[rd] & & sx \ar@{-}[dl] \\ & x & }\qquad 
\xymatrix{ & sw  \ar@{-}[ld] \\ 
w \ar@{-}[rd] &  \\ & x \ar@{-}[dl] \\
sx &  }\qquad
\xymatrix{ & sw \ar@{-}[rd] \ar@{-}[ld]   &\\ 
w \ar@{-}[rd] & & x \ar@{-}[dl] \\ & sx & }
\end{displaymath}
\caption{(i)-(iii) of Proposition \ref{ind}} \label{fig-2}
\end{figure}
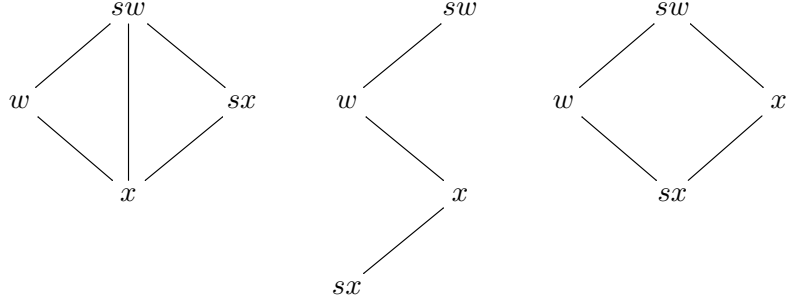

The following corollary is immediate by applying Proposition \ref{ind} (i) and (iii) recursively. Throughout, we let $\Phi_w:=\Phi_+\cap w\Phi_-$ be the inversion set of $w^{-1}$. 

\begin{cor} \label{cor} We have
\[ c_{w,e}={\mathcal T}_w(1) \quad \text{ and } \quad 
c_{w,w}=\prod_{\alpha\in \Phi_w}(1-v e^{-\alpha}).
\]
\end{cor}

\section{Good words and shellability of Bruhat order}

\subsection{Good words}\label{goodword} Following \cite{BN}, we consider the notion of a {good word}. Assume that $x\leq w$, and
introduce the sets
\begin{equation}\label{set}
S(x,w):=\{\alpha\in\Phi_+ \, | \,  x\leq ws_\alpha<w\},\quad R(x,w)=\{s_\alpha \,|\, \alpha\in S(x,w)\}.
\end{equation}
Deodhar's inequality states that 
\begin{equation}\label{deo}
\#S(x,w)=\#R(x,w)\geq \ell(w)-\ell(x)\,,
\end{equation}
with equality holding if the Kazhdan--Lusztig polynomial $P_{w_\circ w,w_\circ x}=1$, or equivalently if the Schubert variety $X_{w_\circ x}$ is rationally smooth at the $T$-fixed point $e_{w_\circ w}$ (see \cite{BL}). We remark that $\#S(x,w)$ has the trivial upper bound $\ell(w)$ because of the inclusion $S(x,w)\subset \Phi_{w^{-1}}=\Phi_+\cap w^{-1}\Phi_-$, where the last set is the inversion set of $w$, of cardinality $\ell(w)$; indeed, it is well known that $\alpha\in\Phi_+$ is an inversion of $w$, i.e., $w\alpha\in\Phi_-$, if and only if $ws_\alpha<w$.

For any reduced expression $\mathfrak{w}=s_1\cdots s_n$ of $w$, let $\lambda_{x,\mathfrak{w}}$ be the set of integers $i\in [1,n]$ such that $x\leq s_1\cdots \hat{s}_i\cdots s_n$. Let $\alpha_i\in\Pi$ be such that $s_i=s_{\alpha_i}$, $i=1,\ldots, n$. Then there are bijections
\[
\lambda_{x,\mathfrak{w}}\to S(x,w)\to R(x,w),\quad i\mapsto \gamma_i:=s_n\cdots s_{i+1}\alpha_i\mapsto s_{\gamma_i}=s_n \cdots s_{i+1}s_i s_{i+1}\cdots s_n.
\]
Moreover it is clear that $ws_{\gamma_i}=s_1\cdots \hat{s}_i\cdots s_n$.
By abuse of notation, we also write
\begin{equation}\label{gc}
\lambda_{x,\mathfrak{w}}=( i_1,\ldots, i_d)\in{\bf N}^d
\end{equation}
for the vector formed by elements of $\lambda_{x,\mathfrak{w}}$
 arranged in ascending order $i_1< \cdots <i_d$. Then $\mathfrak{w}$ is called a \textit{good word} for $x$ if 
\begin{equation}\label{good}
x=s_1\cdots \hat{s}_{i_1}\cdots \hat{s}_{i_d}\cdots s_n.
\end{equation}
Since $d=\#\lambda_{x,\mathfrak{w}}\geq \ell(w)-\ell(x)$, a good word exists only if (\ref{good}) is a reduced expression hence $d=\ell(w)-\ell(x)$.  Conversely, it is conjectured in \cite{BN} that if $W$ is simply-laced and $d=\ell(w)-\ell(x)$, then $w$ has a good word for $x$.
This conjecture is proved in [\textit{loc. cit.}] for $W=A_4$ or $D_4$ using {\sc Sage}, and it is shown to be false in non-simply-laced case, e.g. for $W=B_2$.


\subsection{Shellability} We recall the lexicographic shellability of Bruhat order, following \cite{BW}. For $x, y\in W$, we say that $y$ \textit{covers} $x$, denoted by $y\to x$, if $y>x$ and there is no $z\in W$ such that $y>z>x$. In this case $\ell(y)=\ell(x)+1$ and there is a unique $\alpha\in {\Phi}_+$ such that $s_\alpha y=x$. Moreover for any reduced expression $y=s_1\ldots s_l$, there is a unique $1\leq i\leq l$ such that $x=s_1\cdots \hat{s}_i\cdots s_l$, and one has $\alpha=s_1\cdots s_{i-1}\alpha_i$. We may also write $y\stackrel{\alpha}{\to }x$ to specify the reflection $s_\alpha$ that takes $y$ to $x$.

Consider $x\leq w$ and the Bruhat interval $[x,w]:=\{ y\in W| x\leq y\leq w\}$. Then all maximal chains $\scr{C}: w=w_0\to w_1\to \cdots \to w_d=x$ of $[x,w]$ have the same length $d=\ell(w)-\ell(x)$. Let us describe a labeling of the maximal chains of $[x,w]$. Fix once for all a reduced expression $\mathfrak{w}=s_1\cdots s_n$ of $w$. For a maximal chain ${\scr{C}}$ of $[x,w]$ as above, there is a unique sequence $i_1, \cdots, i_d$ of distinct integers in $[1,n]$ such that $w_k$ is obtained by removing $s_{i_1},\cdots, s_{i_k}$ from $\mathfrak w$, $k=1,\ldots, d$. In particular this implies that the resulting subwords representing $w_k$'s are all reduced. Then we assign the label 
\begin{equation}\label{lab}
\lambda({\scr{C}})=(\lambda_1({\scr{C}}),\ldots,\lambda_d({\scr{C}})):= (i_1,\ldots, i_d)\in {\bf N}^d.
\end{equation}

Recall that the \textit{lexicographic order} of ${\bf N}^d$ is the linear ordering $<_L$ such that ${\bf a}=(a_1,\ldots, a_d)<_L{\bf b}=(b_1,\ldots, b_d)$ if $a_i<b_i$ in the first coordinate where they differ. The main result of \cite{BW} states that $[x,w]$ is lexicographically shellable. In particular this implies that 

(i) there is a unique maximal chain ${\scr{C}}^+_{x,\mathfrak{w}}$ in $[x,w]$ whose label $\lambda({\scr{C}}^+_{x,\mathfrak{w}})$ is increasing, i.e., $\lambda_1({\scr{C}}^+_{x,\mathfrak{w}})<\cdots <\lambda_d({\scr{C}}^+_{x,\mathfrak{w}})$; 

(ii) $\lambda({\scr{C}}^+_{x,\mathfrak{w}})<_L \lambda({\scr{C}})$ for any other maximal chain ${\scr{C}}$ of $[x,w]$. 
\\
Note that the maximal chain ${\scr{C}}^+_{x,\mathfrak{w}}$ depends on the choice of the reduced word $\mathfrak{w}$ which we fix from the beginning. 

Similarly, consider the reduced word $s_n\cdots s_1$ of $w^{-1}$. By applying shellability to $w^{-1}$ with this reduced word and reverting to $w$, we see that 

(i$'$) there is a unique maximal chain ${\scr{C}}^-_{x,\mathfrak{w}}$ in $[x,w]$ whose label $\lambda({\scr{C}}^-_{x,\mathfrak{w}})$ is decreasing, i.e., $\lambda_1({\scr{C}}^-_{x,\mathfrak{w}})> \cdots >\lambda_d({\scr{C}}^-_{x,\mathfrak{w}})$; 

(ii$'$) $\lambda({\scr{C}}^-_{x,\mathfrak{w}})>_L \lambda({\scr{C}})$ for any other maximal chain ${\scr{C}}$ in $[x,w]$.

\section{Main Result} \label{sec-main}

In this section we compute the coefficient $c_{w,x}$, for $x\leq w$, under either of the following two conditions for the pair $(w,x)$:

\medskip

(A) $w$ admits a reduced word $\mathfrak w$ such that $\lambda_{x, \mathfrak w}=\lambda({\scr{C}}^{-}_{x, \mathfrak w})^*=(i_1,\ldots, i_d)$;

(B) $w$ admits a reduced word $\mathfrak w$ such that $\lambda(\scr{C}_{x,\mathfrak w}^+)=
\lambda(\scr{C}_{x,\mathfrak w}^-)^*=(i_1,\ldots, i_d)$.

\medskip

Here we  write $\lambda^*=(i_d,\ldots, i_1)\in{\bf N}^d$ for a vector $\lambda=(i_1,\ldots, i_d)\in {\bf N}^d$.
Note that the reduced word $\mathfrak w$ satisfying Condition (A) is necessarily a good word for $x$. 

As we will prove, both conditions guarantee that only the relations in Proposition~\ref{ind}~(i) and (iii) are used in the recursive computation of $c_{w,x}$; these relations have the advantage of being simple, compared with the relation in part (ii).  

\subsection{Lemmas on good words and shellability} We first prove a few more facts regarding combinatorial properties of a reduced word.

\begin{lemma}\label{6.1}
Assume that $\mathfrak{w}=s_1\cdots s_n$ is a good word of $w$ for $x$ such that $\lambda_{x,\mathfrak w}=(i_1, \ldots, i_d)$ with $i_1>1$. Then 

{\rm (i)} $x\not\leq s_1 w\,;$

{\rm (ii)} $S(s_1x, s_1w)=S(x,w)\,;$

{\rm (iii)} $s_1\mathfrak{w}:=s_2\cdots s_n$ is a good word of $s_1 w$ for $s_1 x$ and 
$\lambda_{s_1x, s_1\mathfrak{w}}=(i_1-1,\ldots, i_d-1)\,$.
\end{lemma}

\begin{proof}
Part (i) is obvious from the definition of good word. Part (iii) follows from (ii). To prove (ii), it suffices to show that $S(s_1x, s_1w)$ is contained in $S(x,w)$, which implies that
$S(s_1x, s_1w)=S(x,w)$  because of Deodhar's inequality
\[
\#S(s_1x,s_1w)\geq \ell(s_1w)-\ell(s_1x)=\ell(w)-\ell(x)=\#S(x,w).
\]
Take $\alpha\in S(s_1x, s_1w)$, i.e., $s_1x\leq s_1ws_\alpha< s_1w$. We claim that
$s_1ws_\alpha < ws_\alpha$. To the contrary, assume that $s_1ws_\alpha>ws_\alpha$. Then by Lemma \ref{lem} we have the diamond square
\begin{displaymath}
\xymatrix{ & s_1w s_\alpha \ar@{--}[rd] \ar@{-}[ld] \ar@{-}[dd]  &\\ 
ws_\alpha \ar@{--}[rd] & & x \ar@{-}[dl] \\ & s_1x & }
\end{displaymath}
where the two dashed lines follow from the middle vertical line.
This implies that $x\leq s_1ws_\alpha <s_1 w$, a contradiction to part (i). Hence $s_1ws_\alpha < ws_\alpha$, and using Lemma \ref{lem} again we obtain the diagram
\begin{displaymath}
\xymatrix{ & w \ar@{-}[rd] \ar@{-}[ld]  & &\\ 
s_1w\ar@{-}[rd] & & ws_\alpha \ar@{-}[dl]  \ar@{--}[dr] \\ & s_1ws_\alpha & & x \\
& &  s_1x \ar@{-}[ul] \ar@{-}[ur]& }
\end{displaymath}
which implies that $\alpha\in S(x,w)$.
\end{proof}

\begin{lemma}\label{6.2}
Let $\mathfrak{w}=s_1\cdots s_n$ be a fixed reduced word of $w$,
$\lambda(\scr{C}^+_{x,\mathfrak{w}})=(i_1,\ldots, i_d)$, $\lambda(\scr{C}^-_{x,\mathfrak{w}})=
(j_d,\ldots, j_1)$, where $i_1<\cdots<i_d$ and $j_1<\cdots <j_d$. Consider the reduced word $s_1\mathfrak{w}=s_2\cdots s_n$ of $s_1w$. Then

{\rm (i)}  if $i_1>1$, then  $x\not\leq s_1w$ and
\[
\lambda(\scr{C}^+_{s_1x, s_1\mathfrak{w}})=(i_1-1,\ldots, i_d-1);
\]

{\rm (ii)} if $j_1>1$, then 
\[
\lambda(\scr{C}^-_{s_1x, s_1\mathfrak{w}})=(j_d-1,\ldots, j_1-1);
\]

{\rm (iii)} if $i_1=1$, then
\[
\lambda(\scr{C}^+_{x,s_1\mathfrak w})=(i_2-1,\ldots, i_d-1);
\]

{\rm (iv)} if $j_1=1$, then $x<s_1x$ and 
\[
\lambda(\scr{C}^-_{x, s_1\mathfrak{w}})=(j_d-1,\ldots, j_2-1).
\]
\end{lemma}

\begin{proof} Write $\scr{C}^\pm_{x,\mathfrak w}: w=w_0^\pm\to w_1^\pm\to \cdots \to w_d^\pm=x$.

(i) The last claim is clear since we have obviously a maximal chain 
\[
{\scr{C}}: s_1 w=s_1w^+_0 \to s_1  w^+_1 \to \cdots \to s_1 w^+_d =s_1 x
\]
of $[s_1x, s_1w]$ with increasing label $\lambda({\scr{C}})=(i_1-1,\ldots, i_d-1)$. We must have ${\scr{C}}=\scr{C}^+_{s_1x, s_1\mathfrak{ w}}$ because of the uniqueness of increasing label. It remains to prove that $x\not\leq s_1 w$. To the contrary, assume that $x\leq s_1w=s_2\cdots s_n$. Then  concatenation of $w\to s_1w$ with any maximal chain in $[x, s_1w]$ will give a maximal chain ${\scr{C}}$ in $[x, w]$ such that ${\scr{C}}<_L \scr{C}^+_{x,\mathfrak{w}}$, since $\lambda_1({\scr{C}})=1< \lambda_1(\scr{C}^+_{x,\mathfrak{w}})=i_1$. This is a contradiction. 

(ii) The proof is similar.

(iii) $\scr{C}^+_{x, s_1\mathfrak w}$ equals the following subchain of $\scr{C}^+_{x,\mathfrak w}$
\[
 s_1 w=w^+_1\to w^+_2\to \cdots \to w^+_d=x.
\]

(iv) $s_1x \rightarrow x$ is the last arrow in the chain $\scr{C}^-_{x,\mathfrak w}$ hence $x<s_1x$. The following subchain of $\scr{C}^-_{x,\mathfrak w}$
\[
w=w^-_0\to w^-_1\to \cdots \to w^-_{d-1}=s_1 x
\]
gives rise to the maximal chain of $[x, s_1 w]$
\[
\scr{C}: s_1w=s_1 w^-_0\to s_1 w^-_1\to \cdots\to s_1 w^-_{d-1}=x
\]
with decreasing label $\lambda(\scr{C})=(j_d-1,\ldots, j_2-1)$, which implies that $\scr{C}=\scr{C}^-_{x,s_1\mathfrak{w}}$.
\end{proof}

\begin{lemma}\label{6.3}
Assume that $\mathfrak{w}=s_1\cdots s_n$ is a good word of $w$ for $x$ such that $\lambda_{x, \mathfrak w}=\lambda(\scr{C}^-_{x,\mathfrak{w}})^*=( i_1, \dots , i_d)$ with $i_1=1$. 
Then 

{\rm (i)} $x<s_1x\,;$

{\rm (ii)} $S(x, s_1w)=S(x,w)\setminus\{\gamma_1\}$, where $\gamma_1=s_n\cdots s_2\alpha_1\,;$

{\rm (iii)} $s_1\mathfrak{w}=s_2\cdots s_n$
is a good word of $s_1w$ for $x\,;$

{\rm (iv)} $\lambda_{x,s_1\mathfrak{w}}=(i_2-1,\ldots, i_d-1)=\lambda(\scr{C}^-_{x,s_1\mathfrak{w}})^*\,$.
\end{lemma}

\begin{proof}
Part (i) and the last equality in (iv) follow from Lemma \ref{6.2} (iv). Part (iii) and the first equality in (iv) are direct consequences of (ii). Finally (ii) follows from 
Lemma \ref{6.4} below, which is of independent interest. \end{proof}

\begin{lemma}\label{6.4}
Assume that $x<w$, $sw<w$ and $x<sx$, where $s=s_\alpha \in S$. If $\# S(x,w)=\ell(w)-\ell(x)$, then  
$S(x, sw)=S(x,w)\setminus \{-w^{-1}\alpha\}$.
\end{lemma}

\begin{proof}
Consider the following diamond given by Lemma \ref{lem}.
\begin{displaymath}
\xymatrix{ & w \ar@{--}[rd] \ar@{-}[ld] \ar@{-}[dd]   &\\ 
sw \ar@{--}[rd] & & sx \ar@{-}[dl] \\ & x & }
\end{displaymath}
Take $\beta\in S(x,sw)$, i.e., $sw>sws_\beta\geq x$. We claim that $\beta\in
S(x,w)$, i.e., $w>ws_\beta \geq x$. If $ws_\beta> sws_\beta$, then the claim is obvious, again by Lemma \ref{lem}. If $ws_\beta< sws_\beta$, then Lemma \ref{lem} gives the following diamond
\begin{displaymath}
\xymatrix{ & sws_\beta \ar@{--}[rd] \ar@{-}[ld] \ar@{-}[dd]   &\\ 
ws_\beta \ar@{--}[rd] & & sx \ar@{-}[dl] \\ & x & }
\end{displaymath}
Hence the claim follows. Obviously $\beta\neq -w^{-1}\alpha\in S(x,w)$, because 
$sws_{w^{-1}\alpha}=w>sw$. Therefore we get an inclusion 
$S(x,sw)\subset S(x,w)\setminus\{-w^{-1}\alpha\}$. This inclusion is an equality because of Deodhar's inequality 
\[
\# S(x,sw)\geq \ell(sw)-\ell(x)=\ell(w)-\ell(x)-1=\# S(x,w)-1,
\]
where the last equality follows from the assumption $\# S(x,w)=\ell(w)-\ell(x)$.
\end{proof}

\subsection{Main theorem} We can now apply previous lemmas together with Proposition \ref{ind} recursively to compute $c_{w,x}$, assuming Condition (A) or (B). As mentioned above, only cases (i) and (iii) of Proposition \ref{ind} show up in the computation. In order to formulate our main result, we introduce an additional notation. 

For any $\alpha\in \Phi$, let
\begin{equation}\label{da}
\partial_\alpha=\frac{1-e^{-\alpha}s_\alpha}{1-e^{-\alpha}},\quad {\mathcal T}_\alpha=(1-ve^{-\alpha})\partial_\alpha-1.
\end{equation}
 Using this notation, it is easy to see that we have 
 \begin{equation}\label{cross}
 w\cdot  \partial_\alpha=  \partial_{w\alpha}\cdot w,\quad
 w\cdot  {\mathcal T}_\alpha= {\mathcal T}_{w\alpha}\cdot w.
 \end{equation}
 
\begin{theorem} \label{thm-main}
Assume that either condition {\rm (A)} or {\rm (B)} holds. In either case, let
\[
\beta_i=s_1\cdots \hat{s}_{i_1}\cdots \hat{s}_{i_2}\cdots \alpha_i,\quad i=1,\ldots, n.
\]
 Then we have 
\[
c_{w,x}=(1-ve^{-\beta_1})\cdots {\mathcal T}_{\beta_{i_1}}\left(\cdots
{\mathcal T}_{\beta_{i_d}}\left(\cdots (1-ve^{-\beta_n})\right)\cdots\right)\,.
\] 
\end{theorem}

\begin{proof}
In either case we use recursion. First assume Condition (A). If $i_1>1$, then $(s_1w, s_1 x)$ satisfies Condition (A) as well, due to Lemma \ref{6.1} (iii) and Lemma \ref{6.2} (ii). Moreover, we may apply Proposition \ref{ind} (iii) because of Lemma \ref{6.1} (i), which gives that
\[
c_{w,x}=(1-ve^{-\alpha_1})s_1(c_{s_1w,s_1x}).
\]
If $i_1=1$, then $(s_1w, x)$  also satisfies (A) and we may apply Proposition \ref{ind} (i), due to Lemma \ref{6.3}, which gives that
\[
c_{w,x}={\mathcal T}_{\alpha_1}(c_{s_1w,x}).
\]
Iterating this process gives us
\begin{align*}
c_{w,x}&=(1-ve^{-\alpha_1})s_1\cdots {\mathcal T}_{\alpha_{i_1}}\left(\cdots {\mathcal T}_{\alpha_{i_d}}\left(\cdots 
(1-ve^{-\alpha_n})s_n(1)\right)\cdots\right)\,.
\end{align*}
One may use (\ref{cross}) to push the reflections $s_i$, $1\leq i \leq n$, $i\neq i_1,\ldots, i_d$ across the operators ${\mathcal T}_{\alpha_{i_1}}$, $\ldots$, ${\mathcal T}_{\alpha_{i_d}}$, noting that $x(1)=1$.

The proof assuming Condition (B) is similar. If $i_1>1$, then by Lemma \ref{6.2} (i)-(ii),
$(s_1w, s_1x)$ also satisfies (B) and Proposition \ref{ind} (iii) applies. If $i_1=1$, then 
by Lemma \ref{6.2} (iii)-(iv), $(s_1w, x)$ satisfies (B) and Proposition \ref{ind} (i) applies.
 \end{proof}

\begin{rmks} (i) Note that in the special cases $d=n$ and $d=0$, we recover $c_{w,e}$ and $c_{w,w}$ respectively, as given by Corollary \ref{cor}.

(ii) The roots $\beta_i$ can be interpreted as follows. We have
\[
\{\beta_i: 1\leq i\leq n, i\neq i_1,\ldots, i_d\}=\Phi_x=\Phi_+\cap x\Phi_-.
\]
Moreover, under Condition (B), the roots $\beta_{i_1},\ldots, \beta_{i_d}$ give the sequence of reflections along the maximal chain $\scr{C}^+_{x,\mathfrak w}$ of $[x,w]$, i.e., we have
 \[
\scr{C}^+_{x,\mathfrak w}: w=w_0\stackrel{\beta_{i_1}}{\to} w_1\to \cdots \stackrel{\beta_{i_d}}{\to} w_d=x.
 \]
 Hence the calculation of $c_{w,x}$ amounts to inserting the operators 
 \[
 {\mathcal T}_\beta=(1-ve^{-\beta})\partial_\beta-1, \quad \beta\in \{\beta_{i_1},\ldots, \beta_{i_d}\}
 \]
  into the product $\prod_{\alpha\in\Phi_x}(1-ve^{-\alpha})$ in a natural, combinatorial way.
\end{rmks}

\subsection{Conditions (A) and (B)}\label{cond-a-b} Since these conditions are essential for our main result, we now discuss them in more detail. We start with an example.

\begin{exa}
Consider $w=s_1s_2s_1s_3s_2s_1$ and $x=s_2s_3$ in $A_3$. It is easy to see that both Conditions (A) and (B) hold in this example.
\end{exa}

Based on thorough computer tests, we now formulate a conjecture about the equivalence of Conditions (A) and (B) in a strong sense. 

\begin{conj} \label{conj} Let $\mathfrak w$ be a reduced word for $w$ and $x\le w$. The following are equivalent:

{\rm (i)} $\lambda_{x, \mathfrak w}=\lambda({\scr{C}}^{-}_{x, \mathfrak w})^*\,;$

{\rm (ii)} $\lambda(\scr{C}_{x,\mathfrak w}^+)=\lambda(\scr{C}_{x,\mathfrak w}^-)^*\,;$

{\rm (iii)} $\lambda_{x, \mathfrak w}=\lambda(\scr{C}_{x,\mathfrak w}^+)\,.$
\end{conj}

The proof of the conjecture, due to D. Muthiah and A. Pusk\'as, is included as an Appendix. Thus, we are able use the more symmetric condition 
\[\lambda_{x, \mathfrak w}=\lambda(\scr{C}_{x,\mathfrak w}^+)=\lambda({\scr{C}}^{-}_{x, \mathfrak w})^*\,.\]
Note that it is enough to prove ${\rm (i)}\Leftrightarrow {\rm (ii)}$, as ${\rm (i)}\Leftrightarrow {\rm (iii)}$ would easily follow; indeed, just reverse the reduced word and use the fact that inversion is an automorphism of the Bruhat order. 

We now discuss some statistics related to the frequency with which Conditions (A) and (B) are satisfied. We looked at the symmetric groups $S_4$, $S_5$, and $S_6$, as well as at the hyperoctahedral groups $B_4$ and $B_5$. For each (signed) permutation $w$, we calculated (with the help of a computer) the percentage of $x\le w$ which satisfy Conditions (A) and (B). The distribution of these percentages in $S_5$ and $S_6$ is shown in Figure~\ref{fig-3}. It is interesting to note that this distribution is skewed right, with the mode at the right tail, while the interquartile range reaches 100\% in both cases. By contrast, in type $B$, the distribution looks closer to a uniform one.

Experiments with the same Weyl groups mentioned above also showed that the formula in Theorem~\ref{thm-main} fails if Conditions (A) and (B) are not satisfied.

\begin{figure}[ht]
\[
\begin{array}{cc}
\!\!\!\!\!\!\!\!\!\!\!\!\!\!\!\!\!\!\!\!\!\!\!\!\!\!\!\!\!\!\!\!\!\!\!\!\!\!\!\!\!\!		\includegraphics[scale=0.75]{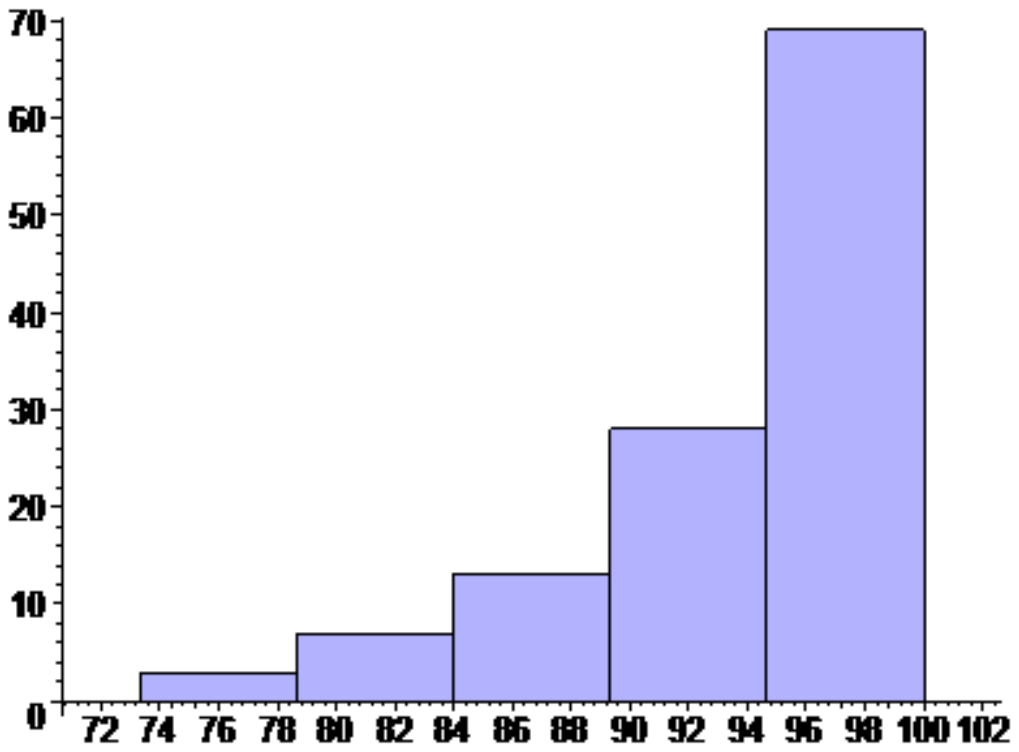} & \!\!\!\!\!\!\!\!\!\!\!\!\!\!\!\!\!\!\!\!\!\!\!\!\!\!\!\!\!\!\!\!\!\!\!\!\!\!\!\!\!\!\!\!\!\!\!\!\!\!\!\!\!\!\!\!\!\!\!\!\!\!\!\!\!\!\!\!\!\!\!\!\!\!\!\!\!\!\!\!\!\!\!\!\!\!\!\!\!\!\!\!\!\!\!\!\!\!\!\!\!\!\!\!\!\!\!\!\!\!\!\!\!\!\!\!\!\!\!\!\!\!\!\!\!\!\!\!\!\!\!\!\!\!\!\! \includegraphics[scale=0.75]{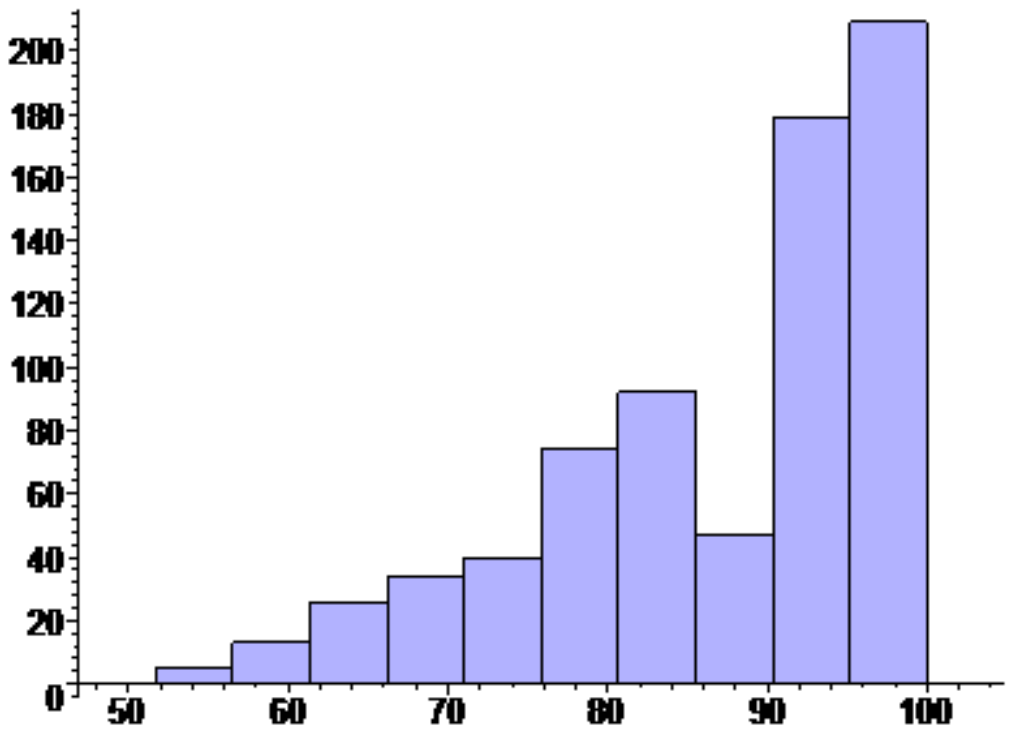}
		\end{array}\]
		\vspace{-13.4cm}
		\caption{Histograms for $S_5$ and $S_6$} \label{fig-3}
\end{figure}

\section{Casselman's basis of Iwahori vectors} \label{sec-Cass}

In this section, under the shellability Condition (B) in Section 6, we compute the transition matrix between two natural bases of the Iwahori fixed vectors in a spherical representation of a semisimple $p$-adic group, considered by Casselman in \cite{C}.  For simply-laced cases, a conjectural formula is given in \cite{BN}, which is proved under the assumption that a good word exists; however, it seems that there is a gap in this proof, which we do not know how to fix at present.  
We follow the strategy of computations in \cite{BN}, although we consider reduced words from a very different point of view. Let us first recall the basic formulations and collect a few results we need from \cite{BN}.

Let $\chi=\chi_{\bf z}$ be an unramified character of $T(F)$, which is parametrized by an element ${\bf z}$ in the  complex torus $\hat{T}$ of  the L-group ${}^LG$. Let $V(\chi)=\textrm{Ind}^G_B\chi$ be the induced representation which consists of locally constant functions $f: G\to{\bf C}$ such that $f(bg)=(\delta^{1/2}\chi)(b) f(g)$, where $\delta=\det(\textrm{Ad}|_\mathfrak{n})$ is the modular character. Let $J$ be the Iwahori subgroup which is the preimage of $B({\bf F}_q)$ under the reduction $K=G(\cal O_F)\to G({\bf F}_q)$. Then the space of $J$-fixed vectors $V(\chi)^J$ has dimension $|W|$, and there are two bases $\{\phi_{w, \chi}\}$ and $\{f_w\}$ of $V(\chi)^J$ parametrized by $W$.

The first natural basis $\{\phi_{w, \chi}\}$ is defined using the disjoint decomposition $G=\bigsqcup_{w\in W}BwJ$ such that
$\phi_w$ is supported on $BwJ$ and $\phi_{w,\chi}|_{wJ}=1$. Let $M_w: V(\chi)\to V({}^w\chi)$ be the intertwining operator defined by
\[
(M_wf)(g)=\int_{N\cap wN_-w^{-1}}f(w^{-1}ng)dn.
\]
Then $\{f_w\}$ is the dual basis of the linear functionals $V(\chi)\to {\bf C}$, $f\mapsto (M_wf)(1)$, $w\in W$. Casselman \cite{C} asks for the transition matrix between these two bases, which is in general a very difficult problem. It is better to use the basis
\[
\psi_{x,\chi}=\sum_{w\geq x}\phi_{w,\chi}
\]
instead of $\phi_{w,\chi}$, and by M\"{o}bus inversion one has
\[
\phi_{x,\chi}=\sum_{w\geq x}(-1)^{\ell(w)-\ell(x)}\psi_{w,\chi}.
\]
If we write $\psi_{x,\chi}=\sum_{w\in W} m(x,w)f_w$, then obviously $m(x,w)=(M_w\psi_{x,\chi})(1)$ and in \cite{BN} it is shown that $(m(x,w))$ is upper triangular.  In [\textit{loc. cit}] it is conjectured that 
\[
m(x,w)=\prod_{\alpha\in S(x,w)}\frac{1-q^{-1}{\bf z}^\alpha}{1-{\bf z}^\alpha}
\]
when the root system $\Phi$ is simply-laced and $|S(x,w)|=\ell(w) - \ell(x)$,  and it is proved under the additional assumption that $w$ admits a good word for $x$.

Let $H$ be the Iwahori--Hecke algebra which consists of bi-$J$-invariant functions supported on $K$. Then $H$ has a basis $\{t_w| w\in W\}$, where $t_w$ is the characteristic function of $JwJ$, and $H$ is generated by $t_i:=t_{\sigma_i}$, $i\in I$. Let $\alpha_\chi: V(\chi)^J\to H$ be the isomorphism of left $H$-modules defined by 
$(\alpha_\chi f)(g)=\left. f(g^{-1})\right|_K$. Let $\cal M_w=\cal M_{w, {\bf z}}: H\to H$ be the map making the following diagram commute:
\begin{displaymath}
\xymatrix{
V(\chi)\ar[r]^{M_w} \ar[d]^{\alpha_\chi} & V({}^w\chi) \ar[d]^{\alpha_{{}^w\chi}} \\
H \ar[r]^{\cal M_{w}} & H 
}
\end{displaymath}
Define $\mu_{\bf z}(w)=\cal M_w(1_H)\in H$. Then 
\begin{equation}\label{mu1}
\mu_{\bf z}(\sigma_i)=q^{-1}t_i +(1-q^{-1})\frac{{\bf z}^{a_i}}{1-{\bf z}^{a_i}},
\end{equation}
and for $\ell(w_1w_2)=\ell(w_1)+\ell(w_2)$ one has
\begin{equation}\label{mu2}
\mu_{\bf z}(w_1w_2)=\mu_{\bf z}(w_2)\mu_{w_2{\bf z}}(w_1).
\end{equation}
Define $\psi(x)=\alpha_\chi(\psi_x)\in H$. Then $\psi(x)=\sum_{w\geq x}t_w$ is independent of $\chi$. For $f\in H$ let $\Lambda(f)$ be the coefficient of $1$ in the expression of $f$ in terms of the basis $t_w$. Then 
 \[
 m(x,w)=\Lambda(\psi(x)\mu_{\bf z}(w)).
 \]
 For $f, g\in H$ and $x\in W$, write $f-g\geq x$ if $f-g$ is a linear combination of $t_w$'s with $w\geq x$.
 
 \begin{prop}\cite{BN} \label{bn}
 Let $s=s_\alpha\in S$, $x\in W$ such that $xs>x$. Then
 \[
 \psi(x)\mu_{\bf z}(s)=\frac{1-q^{-1}{\bf z}^\alpha}{1-{\bf z}^\alpha}\psi(x),\quad \psi(xs)\mu_{\bf z}(s)- \psi(x) \geq xs. 
 \]
 \end{prop}
 
 Now we can give our formula for $m(x,w)$ in full root system generality, assuming that Condition (B) holds.
 
 \begin{theorem}\label{7.2}
 Assume condition {\rm (B)}. Let $\gamma_{i_k}= s_n\cdots s_{i_k+1}\alpha_{i_k}$, $k=1,\ldots, d$. Then
 \[
 m(x,w)=\prod^d_{k=1}\frac{1-q^{-1}{\bf z}^{\gamma_{i_k}}}{1-{\bf z}^{\gamma_{i_k}}}.
 \]
 \end{theorem}

\begin{proof} The proof follows the argument in \cite{BN}, but we shall give some details for the sake of completeness. Write $\mu(s_n)=\mu_{\bf z}(s_n)$, $\mu(s_{n-1})=\mu_{s_n({\bf z})}(s_{n-1}), \ldots$, suppressing the dependence of spectral parameters.
Write $\psi(x)\mu_{\bf z}(w)$ as a sum
\begin{align*}
[\psi(s_1\cdots\hat{s}_{i_1}\cdots \hat{s}_{i_d}\cdots s_n)\mu(s_n)-\psi(s_1\cdots\hat{s}_{i_1}\cdots \hat{s}_{i_d}\cdots s_{n-1})]\mu(s_{n-1})\cdots\mu(s_1)&+\\
[\psi(s_1\cdots\hat{s}_{i_1}\cdots \hat{s}_{i_d}\cdots s_{n-1})\mu(s_{n-1})-\psi(s_1\cdots\hat{s}_{i_1}\cdots \hat{s}_{i_d}\cdots s_{n-2})]\mu(s_{n-2})\cdots\mu(s_1)&+\\
\cdots & \\
[\psi(s_1\cdots\hat{s}_{i_1}\cdots \hat{s}_{i_d})\mu(s_{i_d})-C(d)\psi(s_1\cdots\hat{s}_{i_1}\cdots \hat{s}_{i_d})]\mu(s_{i_d-1})\cdots\mu(s_1) &+\\
C(d)[\psi(s_1\cdots\hat{s}_{i_1}\cdots s_{i_d-1})\mu(s_{i_d-1})-\psi(s_1\cdots\hat{s}_{i_1}\cdots s_{i_d-2})]\mu(s_{i_d-2})\cdots\mu(s_1)& +\\
\cdots & \\
C(d)\cdots C(1)[\psi(s_1)\mu(s_1)-\psi(1)]&+\\
C(d)\cdots C(1)\psi(1),&
\end{align*}
where
\[
C(k)=\frac{1-q^{-1}{\bf z}^{\gamma_k}}{1-{\bf z}^{\gamma_k}},\quad k=1,\ldots, d.
\]
We will show that the linear functional $\Lambda$ annihilates every summand except the last, so that $m(x,w)=C(d)\cdots C(1)$.

Since we have the reduced words $w_k^+:=s_1\cdots \hat{s}_{i_1}\cdots \hat{s}_{i_k} s_{i_k+1}\cdots s_n$, $k=1,\ldots,n$, which form the maximal chain ${\scr{C}}^+_{x,\mathfrak{w}}$, we see that $s_1\cdots \hat{s}_{i_1}\cdots s_{i_k}> s_1\cdots \hat{s}_{i_1}\cdots \hat{s}_{i_k}$. Therefore by Proposition \ref{bn} the summands of the form
\[
\prod_{j>k}C(j)[\psi(s_1\cdots \hat{s}_{i_1}\cdots \hat{s}_{i_k})\mu(s_{i_k})-C(k)
\psi(s_1\cdots \hat{s}_{i_1}\cdots \hat{s}_{i_k})]\mu(s_{i_k-1})\cdots \mu(s_1)
\]
are all equal to zero. Note that the spectral parameter of $\mu(s_{i_k})$ is $s_{i_k+1}\cdots s_n{\bf z}$ and one has
$(s_{i_k+1}\cdots s_n{\bf z})^{\alpha_{i_k}}={\bf z}^{s_n\cdots s_{i_k+1}\alpha_{i_k}}={\bf z}^{\gamma_{i_k}}.$

Every other summand is a constant multiple of the form
\begin{equation}\label{other}
[\psi(s_1\cdots \hat{s}_{i_1}\cdots \hat{s}_{i_2}\cdots s_j)\mu(s_j)-\psi(s_1\cdots \hat{s}_{i_1}\cdots\hat{s}_{i_2}\cdots s_{j-1})]\mu(s_{j-1})\cdots \mu(s_1).
\end{equation}
Since $s_1\cdots \hat{s}_{i_1}\cdots \hat{s}_{i_2}\cdots s_j$ is reduced, by Proposition \ref{bn} we have
\[
\psi(s_1\cdots \hat{s}_{i_1}\cdots \hat{s}_{i_2}\cdots s_j)\mu(s_j)-\psi(s_1\cdots \hat{s}_{i_1}\cdots\hat{s}_{i_2}\cdots s_{j-1})\geq s_1\cdots \hat{s}_{i_1}\cdots \hat{s}_{i_2}\cdots s_j.
\]
Applying (\ref{mu1}), (\ref{mu2}) and arguing as in \cite{BN} one can deduce that (\ref{other}) is annihilated by $\Lambda$ unless 
\begin{equation} \label{eqn-no}
s_1\cdots \hat{s}_{i_1}\cdots \hat{s}_{i_2}\cdots s_j \leq s_1\cdots s_{j-1} =s_1 \cdots s_{j-1} \hat{s}_j .
\end{equation}
Assume that \eqref{eqn-no} is true, let $d'=\max\{1\leq k\leq d: i_k<j\}$, $x'=s_1\cdots \hat{s}_{i_1}\cdots\hat{s}_{i_{d'}}\cdots s_j$ and $\mathfrak w'=s_1\cdots s_j$. Recall that we have the reduced words $w_k^-:=s_1\cdots \hat{s}_{i_{d-k+1}}\cdots \hat{s}_{i_d}\cdots s_n$,
$k=1,\ldots, d$, which make the maximal chain $\scr{C}^-_{x,\mathfrak w}$. Consider the following subchain of 
$\scr{C}^-_{x,\mathfrak w}$
\[
w^-_{d-d'}\to w^-_{d-d'+1}\to\cdots \to  w^-_d = x.
\]
By taking reduced subwords, it gives rise to a maximal chain of $[x',\mathfrak w']$
\[
\scr{C}: \mathfrak w'=w_0' \to w_1'\to \cdots \to w_{d'}'= x'
\]
where $w_i'=s_1\cdots \hat{s}_{d'-i+1}\cdots \hat{s}_{d'}\cdots s_j$, $i=1,\ldots, d'$.
Then $\lambda(\scr{C})=(i_{d'},\ldots, i_1)$ is decreasing, which implies that 
$\scr{C}=\scr{C}^-_{x', \mathfrak w'}$. But similarly to the proof of Lemma \ref{6.2} (i), this contradicts (\ref{eqn-no}) because $\scr{C}^-_{x', \mathfrak w'}$ is lexicographically maximal.
This finishes the proof of the theorem.
\end{proof}

\begin{rmk} Given the equivalence of Conditions (A) and (B), proved in the Appendix, the Bump-Nakasuji result \cite{BN} in full root system generality immediately follows from Theorem~\ref{7.2}. 
\end{rmk}

\section{Appendix: Proof of Conjecture \ref{conj} \vskip .2 cm By Dinakar Muthiah and Anna Pusk\'as } \label{appendix}

In this appendix, we prove Conjecture \ref{conj}. The conjecture is that the following three conditions are equivalent. 
\begin{enumerate}[(i)]
\item $\lap{x}{\w}=\ldechr{x}{\w}\,;$
\item $\linch{x}{\w}=\ldechr{x}{\w}\,;$
\item $\lap{x}{\w}=\linch{x}{\w}\,.$
\end{enumerate}
As mentioned right below Conjecture \ref{conj}, it suffices to prove that ${\rm (i)}$ and ${\rm (ii)}$ are equivalent.

We will keep the notations in the previous sections. In particular, we have $x\leq w$ two elements of $W,$ $\w=s_1\cdots s_n$ a reduced word for $w;$ $\lap{x}{\w}=(\la_1,\ldots ,\la _k),$ $\lambdaCplus{x}{\fw} = (i_1, \cdots, i_d)$ and $\lambdaCminus{x}{\fw}^{\ast} = (j_1, \cdots, j_d).$

\subsection{Proof of ${\rm (i)} \implies {\rm (ii)}$}

\begin{lemma}
  \label{lem:ione-equals-one-is-equivalent-to-lambdaone-equals-one}
 Let $x$, $w$, $\fw$ be as before, $\lambda_{x,\fw} = (\lambda_1, \cdots, \lambda_k)$, and $\lambdaCplus{x}{\fw} = (i_1, \cdots, i_d)$. Then $i_1 = 1$ if and only if $\lambda_1 = 1$.
\end{lemma}
\begin{proof}
Assume first that $i_1=1.$ Then the chain $\inch{x}{\w}$ starts with $\w_{\widehat{1}},$ and hence $x\leq \w_{\widehat{1}}$ and thus $\la_1=1.$
For the other direction, assume $\la_1=1.$ Omitting the first simple reflection from $\w$ only decreases its length by $1,$ hence $\ell(\w_{\widehat{1}})=\ell(w)-1.$ Composing $w\rightarrow \w_{\widehat{1}}$ with a maximal chain from $\w_{\widehat{1}}$ to $x$ gives a maximal chain ${\mathscr{C}}$ from $\w$ to $x$ whose label starts with $1.$ Then $\inch{x}{\w}\leq_L {\mathscr{C}}$ implies $i_1=1.$
\end{proof}

\begin{rmk}\label{rmk:wgoodif1}
If ${\rm (i)}$ holds for $x$ and $\w,$ i.e. $\lap{x}{\w}=\ldechr{x}{\w},$ then $\w$ is a good word of $w$ for $x.$ (Omitting all the reflections from $\w$ that appear in $\lap{x}{\w}$ is the same as taking the last element of the maximal chain $\dech{x}{\w};$ that last element is $x.$)
\end{rmk}

\begin{prop}{${\rm (i)} \implies {\rm (ii)}$.}
\end{prop}

\begin{proof}
We proceed by induction on $\ell(w) + (\ell(w) - \ell(x))$; the base case is trivial.  

Assume that ${\rm (i)}$ holds for a pair $x,\w$, i.e. $\lap{x}{\w}=\ldechr{x}{\w}$.  We would like to show that ${\rm (ii)}$ holds for $x,\w$ as well, i.e. $\linch{x}{\w}=\ldechr{x}{\w}$.

Consider $\la_1=j_1,$ the first index in the labels $\lap{x}{\w}=\ldechr{x}{\w}.$ We distinguish between two cases according to whether $\la_1=j_1=1$ or $\la_1=j_1>1.$ 

\noindent {\bf Case 1:} $\la_1=j_1=1.$ Then by Lemma \ref{6.3} (iv), we have that ${\rm (i)}$ holds for the pair $x,\w'=s_1\w$. Then by induction, ${\rm (ii)}$ holds for $x$ and $\w'$, i.e. $\linch{x}{\w'}=\ldechr{x}{\w'}.$ 

By Lemma \ref{6.2} (iii) and (iv), we have:
\begin{equation}
  \label{eq:11}
(i_2-1,\ldots ,i_d-1)=(j_2-1,\ldots ,j_d-1)\,.
\end{equation}
Together with $i_1=1$ (Lemma \ref{lem:ione-equals-one-is-equivalent-to-lambdaone-equals-one}) we conclude that $i_r=j_r$ for every $1\le r\leq d$, hence ${\rm (ii)}$ holds for the pair $x,\w$.

\noindent {\bf Case 2:} $\la_1=j_1>1$. By Lemma \ref{6.1} (iii) and Lemma \ref{6.2} (ii), ${\rm (i)}$ holds for the pair $x' = s_1 x$ and $\w' = s_1 \w$. By induction, ${\rm (ii)}$ also holds for $x',\w'$. By Lemma \ref{lem:ione-equals-one-is-equivalent-to-lambdaone-equals-one}, $i_1 > 1$. Thus by Lemma \ref{6.2} (i) and (ii), we have $(i_1 - 1, \cdots, i_d -1) = (j_1 - 1, \cdots, j_d -1)$, which implies $(i_1, \cdots, i_d) = (j_1, \cdots, j_d )$.
\end{proof}

\subsection{Proof of ${\rm (ii)} \implies {\rm (i)}$}

\begin{lemma}
  \label{lem:remove-one-node-lambda-result}
 Let $x$, $w$, $\fw$ be as before.  Write $\lambdaCplus{x}{\fw} = (i_1, \cdots, i_d)$, $\lambdaCminus{x}{\fw}^* = (j_1, \cdots, j_d)$. 
 
 Suppose $j_1 = 1$, then:
 \begin{itemize}
 \item $x \leq s_1 w\,;$
 \item $\lambda_{x, s_1\fw} = (\lambda_{x,\fw} \backslash \{1\}) - 1\,.$
 \end{itemize}
 
 Suppose $i_1 > 1$, then:
 \begin{itemize}
 \item $s_1 x \leq s_1w\,;$
 \item $\lambda_{s_1 x, s_1\fw} \supseteq \lambda_{x,\fw} - 1\,.$
 \end{itemize}
Here we write $(\lambda_{x,\fw} \backslash \{1\}) - 1$ and $\lambda_{x,\fw} - 1$ to refer to the set obtained by subtracting $1$ from all elements.
\end{lemma}
\begin{proof}
Note that $s_1\fw$ is a reduced word for $s_1w,$ and $s_1w<w.$

First suppose $j_1=1.$ By Lemma \ref{6.2} (iv) we have $x<s_1x.$ By Lemma \ref{lem} we may draw the diagram 
\begin{center}
\begin{tikzpicture}
 \draw[-] (-\picsize,0) -- (0,\picsize);
 \draw[-] (0,-\picsize) -- (\picsize,0);
 \draw[-,dashed] (0,-\picsize) -- (-\picsize,0); 
 \draw[-] (0,-\picsize) -- (0,\picsize); 
  \draw (-0.3*\picsize,-0.2*\picsize) node [] {\large{\rotatebox[origin=c]{45}{$\Longleftarrow$}}}; 
 \draw (-\picsize,0) node [fill=white] {$s_1w$}; 
 \draw (0,\picsize) node [fill=white] {$w$}; 
 \draw (0,-\picsize) node [fill=white] {$x$}; 
 \draw (\picsize,0) node [fill=white] {$s_1x$}; 
\end{tikzpicture}
\end{center}
and conclude that $x \leq s_1 w.$ Let $1<t\leq n$ and $\w_{\widehat{t}}:=s_1\cdots \widehat{s_t}\cdots s_n,$ and $s_1\w_{\widehat{t}}:=s_2\cdots \widehat{s_t}\cdots s_n(=(s_1\w)_{\widehat{t-1}}).$ To show  $\lambda_{x, s_1\fw} = (\lambda_{x,\fw} \backslash \{1\}) - 1,$ it suffices to prove 
\begin{equation}\label{eq:toprove_1}
x\leq \w_{\widehat{t}} \;\;\Longleftrightarrow\;\; x\leq s_1\w_{\widehat{t}}\:.
\end{equation}
(Note that we are slightly abusing notation. For example, when we write $x\leq \w_{\widehat{t}}\:$, we mean $x \leq w_{\widehat{t}}$ where $w_{\widehat{t}}$ is the Weyl group element obtained by multiplying out the word $\w_{\widehat{t}}\:$.)

To prove \eqref{eq:toprove_1}, we use Lemma \ref{lem} again. We have either $\w_{\widehat{t}}<s_1\w_{\widehat{t}}$ or $\w_{\widehat{t}}>s_1\w_{\widehat{t}}\:;$ we may accordingly draw one of the following two diagrams.

\begin{center}
\hfill 
\tikz[anchor=center,baseline]\node{\begin{tikzpicture}
 \draw[-] (-\picsize,0) -- (0,\picsize);
 \draw[-] (0,-\picsize) -- (\picsize,0);
 \draw[-,dashed] (0,-\picsize) -- (-\picsize,0); 
 \draw[-,dashed] (0,-\picsize) -- (0,\picsize); 
 \draw (-0.3*\picsize,-0.2*\picsize) node [] {\large{\rotatebox[origin=c]{45}{$\Longleftrightarrow$}}}; 
 \draw (-\picsize,0) node [fill=white] {$\w_{\widehat{t}}$}; 
 \draw (0,\picsize) node [fill=white] {$s_1\w_{\widehat{t}}$}; 
 \draw (0,-\picsize) node [fill=white] {$x$}; 
 \draw (\picsize,0) node [fill=white] {$s_1x$}; 
\end{tikzpicture}}; \hfill\tikz[anchor=center,baseline]\node{\begin{tikzpicture}
 \draw[-] (-\picsize,0) -- (0,\picsize);
 \draw[-] (0,-\picsize) -- (\picsize,0);
 \draw[-,dashed] (0,-\picsize) -- (-\picsize,0); 
 \draw[-,dashed] (0,-\picsize) -- (0,\picsize); 
 \draw (-0.3*\picsize,-0.2*\picsize) node [] {\large{\rotatebox[origin=c]{45}{$\Longleftrightarrow$}}}; 
 \draw (-\picsize,0) node [fill=white] {$s_1\w_{\widehat{t}}$}; 
 \draw (0,\picsize) node [fill=white] {$\w_{\widehat{t}}$}; 
 \draw (0,-\picsize) node [fill=white] {$x$}; 
 \draw (\picsize,0) node [fill=white] {$s_1x$}; 
\end{tikzpicture}};\hfill 
\end{center}
These diagrams together imply that \eqref{eq:toprove_1} holds in both cases.

Next suppose $i_1>1.$ The argument in this case is very similar to the one above. We claim $x>s_1x.$ Assume to the contrary that $x<s_1x.$ Then again by Lemma \ref{lem} we may draw the following diagram.
\begin{center}
\begin{tikzpicture}
 \draw[-] (-\picsize,0) -- (0,\picsize);
 \draw[-] (0,-\picsize) -- (\picsize,0);
 \draw[-,dashed] (0,-\picsize) -- (-\picsize,0); 
 \draw[-] (0,-\picsize) -- (0,\picsize); 
 \draw (-0.3*\picsize,-0.2*\picsize) node [] {\large{\rotatebox[origin=c]{45}{$\Longleftarrow$}}}; 
 \draw (-\picsize,0) node [fill=white] {$s_1w$}; 
 \draw (0,\picsize) node [fill=white] {$w$}; 
 \draw (0,-\picsize) node [fill=white] {$x$}; 
 \draw (\picsize,0) node [fill=white] {$s_1x$}; 
\end{tikzpicture}
\end{center}
This contradicts the statement of Lemma \ref{6.2} (i) that $x\nleq s_1w.$ Hence we have $x>s_1x,$ and consequently the diagram 
\begin{center}
\begin{tikzpicture}
 \draw[-] (-\picsize,0) -- (0,\picsize);
 \draw[-] (0,-\picsize) -- (\picsize,0);
 \draw[-,dashed] (0,-\picsize) -- (-\picsize,0); 
 \draw (0,0) node [] {\large{\rotatebox[origin=c]{45}{$\Longleftarrow$}}}; 
 \draw[-] (\picsize,0) -- (0,\picsize); 
 \draw (-\picsize,0) node [fill=white] {$s_1w$}; 
 \draw (0,\picsize) node [fill=white] {$w$}; 
 \draw (0,-\picsize) node [fill=white] {$s_1x$}; 
 \draw (\picsize,0) node [fill=white] {$x$}; 
\end{tikzpicture}
\end{center}
shows that $s_1 x \leq s_1w.$ Take $1<t\leq n$ and $\w_{\widehat{t}}$ and $s_1\w_{\widehat{t}}$ as in the case $i_1=1$ above. To prove $\lambda_{s_1 x, s_1\fw} \supseteq \lambda_{x,\fw} - 1$ we need to show 
\begin{equation}\label{eq:toprove_not1}
x\leq \w_{\widehat{t}} \;\;\Longrightarrow\;\; s_1x\leq s_1\w_{\widehat{t}}\:.
\end{equation}
First consider the case when $\w_{\widehat{t}}<s_1\w_{\widehat{t}}\:.$ Then if $x\leq \w_{\widehat{t}}$ we have 
\begin{equation}
s_1x< x\leq \w_{\widehat{t}}<s_1\w_{\widehat{t}}\:,
\end{equation}
whence $s_1x<s_1\w_{\widehat{t}}\:.$
If on the other hand $\w_{\widehat{t}}>s_1\w_{\widehat{t}}\:,$ then again by Lemma \ref{lem} we have the diagram 
\begin{center}
\begin{tikzpicture}
 \draw[-] (-\picsize,0) -- (0,\picsize);
 \draw[-] (0,-\picsize) -- (\picsize,0);
 \draw[-,dashed] (0,-\picsize) -- (-\picsize,0); 
 \draw (0,0) node [] {\large{\rotatebox[origin=c]{45}{$\Longleftrightarrow$}}}; 
 \draw[-,dashed] (\picsize,0) -- (0,\picsize); 
 \draw (-\picsize,0) node [fill=white] {$s_1\w_{\widehat{t}}$}; 
 \draw (0,\picsize) node [fill=white] {$\w_{\widehat{t}}$}; 
 \draw (0,-\picsize) node [fill=white] {$s_1x$}; 
 \draw (\picsize,0) node [fill=white] {$x$}; 
\end{tikzpicture}
\end{center}
which proves \eqref{eq:toprove_not1}. 
\end{proof}

\begin{prop}{${\rm (ii)} \implies {\rm (i)}$.}
\end{prop}

\begin{proof}
Let $x,w, \fw$ be as before.
We proceed by induction on $\ell(w) + (\ell(w) - \ell(x))$; the base case is trivial. 

Let us assume $\lambdaCplus{x}{\fw} = \lambdaCminus{x}{\fw}^*$. Write $\lambda_{x,\fw} = (\lambda_1, \cdots, \lambda_k)$,  $\lambdaCplus{x}{\fw} = (i_1, \cdots, i_d)$, and $\lambdaCminus{x}{\fw} = (j_d, \cdots, j_1)$. Our assumption means that $i_r = j_r$ for all $r$.  

\noindent {\bf Case 1:} $i_1 = j_1 = 1$. In this case $\lambda_1 =1$ by Lemma \ref{lem:ione-equals-one-is-equivalent-to-lambdaone-equals-one}, and Lemma \ref{lem:remove-one-node-lambda-result} tells us that $x \leq s_1w$ and:
\begin{equation}
  \label{eq:01}
\lambda_{x, s_1\fw} = (\lambda_{x,\fw} \backslash \{1\}) - 1\,.
\end{equation}

 Then by Lemma \ref{6.2} (iii) and (iv), we have that:
 \begin{equation}
   \label{eq:3}
\lambdaCplus{x}{s_1\fw} = \lambdaCminus{x}{s_1\fw}^* = (i_2-1,\cdots, i_d -1) \,.  
 \end{equation}
By induction, we know: 
\begin{equation}
  \label{eq:10}
\lambda_{x,s_1\fw} = \lambdaCminus{x}{s_1\fw}^*  \,.
\end{equation}
By \eqref{eq:01} \eqref{eq:3} and \eqref{eq:10}, $\lambda_{x,\fw} = (i_1, \cdots, i_d)$. Therefore $\lambda_{x,\fw} = \lambdaCminus{x}{\fw}^*$.  



\noindent {\bf Case 2:} $i_1 = j_1 > 1$. The argument is very similar. In this case, $\lambda_1 > 1$ by Lemma \ref{lem:ione-equals-one-is-equivalent-to-lambdaone-equals-one}, and Lemma \ref{lem:remove-one-node-lambda-result} tells us that $s_1x \leq s_1w$ and:
\begin{equation}
  \label{eq:5}
\lambda_{s_1x, s_1\fw} \supset \lambda_{x,\fw} - 1  \,.
\end{equation}

By Lemma \ref{6.2} (i) and (ii), we have that:
\begin{equation}
  \label{eq:6}
 \lambdaCplus{s_1x}{s_1\fw} = \lambdaCminus{s_1 x}{s_1\fw}^* = (i_1-1,\cdots, i_d -1) \,.
\end{equation}
 By induction, we know:
 \begin{equation}\label{eq:indconcl}
 \lambda_{s_1x,s_1\fw} = \lambdaCminus{s_1x}{s_1\fw}^*\,.
 \end{equation}
 In particular:
 \begin{equation}
   \label{eq:7}
\# \lambda_{s_1x,s_1\fw}  = \ell(w) - \ell(x)   \,.
 \end{equation}
By Deodhar's inequality:
\begin{equation}
  \label{eq:8}
\# \lambda_{x,\fw} \geq \ell(w) - \ell(x)  \,.
\end{equation}
So (\ref{eq:5}), (\ref{eq:7}), and (\ref{eq:8}) together imply:
\begin{equation}
  \label{eq:9}
 \lambda_{s_1x, s_1\fw} = \lambda_{x,\fw} - 1 \,.
\end{equation}
By \eqref{eq:6}, \eqref{eq:indconcl} and \eqref{eq:9}, $\lambda_{x,\fw} = (i_1, \cdots, i_d)$. Therefore $\lambda_{x,\fw} = \lambdaCminus{x}{\fw}^*$. 
\end{proof}

\end{document}